\documentclass[12pt,twoside,a4paper]{article}
\usepackage[english,catalan,galician,activeacute]{babel}
\usepackage[T1]{fontenc}
\usepackage{inputenc}
\usepackage{amsmath,bm,mathrsfs}
\usepackage{amssymb}
\usepackage{amsxtra}
\usepackage{amsthm}
\usepackage{epic}
\usepackage{eepic} 
\usepackage{fancybox}
\usepackage{boxedminipage}
\usepackage{calc}
\usepackage{ifthen}
\usepackage{paralist}
\usepackage{exscale,relsize}
\usepackage{pgf}
\usepackage[rm,sc,small,center]{titlesec}
\titlelabel{\thetitle.\enspace}
\usepackage{tikz}
\usetikzlibrary{arrows,automata,backgrounds,calendar,snakes}
\usepackage{inputenc}
\usepackage{color}

\newtheorem{lemma}{Lemma}[section]

\newcommand{\R}{\mathbb{R}}

\newcommand{\mA}{\mathcal A}
\newcommand{\mH}{\mathcal H}
\newcommand{\mV}{\mathcal V}
\newcommand{\mR}{\mathcal R}
\newcommand{\bxi}{{\boldsymbol\xi}}

\newcommand{\be}{\begin{eqnarray}}
\newcommand{\ee}{\end{eqnarray}}
\newtheorem{prop}{Proposition}
\newtheorem{coro}{Corollary}
\begin{document}
\selectlanguage{english}
\thispagestyle{empty}
\numberwithin{equation}{section}
%%%%%%%%%%%%%%%%%%%%%%%%%%%%%%%%%%%%%%%%%%%%%%%%%%%%%%%%%%%%%%%%%%%%%%%%%%%%%
%%%%%%%%%%%%%%%%%%%%%%%%%%%%%%%%%%%%%%%%%%%%%%%%%%%%%%%%%%%%%%%%%%%%%%%%%%%%%
\pagestyle{myheadings}
\markboth{\small H. Freist\"uhler}{\small Linear Degeneracy of Second-Order Hyperbolic Systems}
%%%%%%%%%%%%%%%%%%%%%%%%%%%%%%%%%%%%%%%%%%%%%%%%%%%%%%%%%%%%%%%%%%%%%%%%%%%%%
%%%%%%%%%%%%%%%%%%%%%%%%%%%%%%%%%%%%%%%%%%%%%%%%%%%%%%%%%%%%%%%%%%%%%%%%%%%%%
\title{Linear Degeneracy in a Class of\\ Nonlinear Second-Order Hyperbolic Systems} 
\author{\it Heinrich Freist\"uhler\thanks{Department of Mathematics, University of Konstanz, 
78457 Konstanz, Germany; Supported by 
DFG Grant No.\ FR 822/11-1} 
%~ {\rm and}\ \  Blake Temple\thanks{Department of Mathematics, 
%University of California, Davis, Davis CA, USA 95616}} 
%; Supported by NSF Applied Mathematics Grant No.\ DMS-070-7532.}
}
\selectlanguage{english}
\date{June 1, 2024}
\maketitle
%%
%%
%\begin{abstract}
%\end{abstract}
\newpage

\section{Formation of singularities in first-order systems}
A most prominent feature of quasilinear hyperbolic systems  
\be
V_t+A(V)V_x=0, 
\ee 
$A:\Omega\to\R^{n\times n}, \Omega\subset\R^n$,
is the property that many smooth solutions to them develop singularities in finite time. 
Such solutions exist whenever the system has a genuinely nonlinear mode, i.e., 
there is a smooth (eigenvalue,eigenvector) pair 
$(\lambda,R):\Omega^*\to\R\times\R^n \setminus\{0\}$,
\be\label{mode}
(-\lambda(V)I+A(V))R(V)=0,\quad V\in \Omega^*,
\ee
defined on an open neighborhood $\Omega^*\subset\Omega$ 
of some state
$V^*\in\Omega$ with  
\be\label{GNL}
(R(V)\cdot\nabla\lambda(V))|_{V=V^*}\neq 0.
\ee    
The condition \eqref{GNL} of \emph{genuine nonlinearity} goes back to Lax \cite{L57}, and first general results on 
singularity formation to him \cite{LSing}, John \cite{J}, and Liu \cite{Liu}.
For inhomogeneous systems %of \emph{balance laws} 
\be
V_t+A(V)V_x=G(V), 
\ee 
with source $G:\Omega\to\R^n$, 
%one might wonder whether the same holds notably if the reference state is an 
%equilibrium, i.e., satisfies
%
%an appropriate norm of $U_0(\cdot)-U_*$ is sufficiently small lead to solutions that remain smooth for all times $t>0$
%cite{D,KY}. 
B\"arlin has recently shown \cite{B} that  {\it 
%\begin{prop}[\cite{B}]
under the sole assumptions that
\be\label{EL}
G(V^*)=0
\ee
and 
\eqref{GNL} hold, there exist, for any Euclidean neighborhood $\Omega^*$ of $V^*$,
smooth data $V_0:\R\to \Omega^*$ such that the associated regular solution, while its values 
stay in $\Omega^*$, 
ceases to be differentiable after finite time.}   
%\end{prop}
\par\medskip
The present note has been prompted by the fact that this classical mechanism of singularity 
formation due to genuine nonlinearity will generically transfer to second-order hyperbolic systems.
However, we show that this is not the case for a class of systems of the special form
$$ %\be\label{sosn}
B^{00}(U)U_{tt}+C^j(U)U_{tx^j}+B^{jk}(U)U_{x^jx^k}=H(U,U_t,U_x).
$$ %\ee
Systems of this form have recently been introduced as models for dissipative relativistic fluid dynamics,
and an application of the main result of this paper -- Proposition 1 below -- is to provide an indication
that solutions to these models may avoid singularity formation, similarly to what is generically 
expected for the hyperbolic-parabolic models of dissipative fluid dynamics such as the  
Navier-Stokes-Fourier system for compressible fluids \cite{K}. 

\newpage

\section{Dispersion relations and modes}
We begin our study with a linear constant-coefficients system
\be\label{sos}
B^{00}U_{tt}+C^jU_{tx^j}+B^{jk}U_{x^jx^k}=0
\ee
for $U:[0,\infty)\times\R^d\to\R^n$,
where $B^ {00}, C^j, B^{jk}, j,k=1,...,d,$ are real $n\times n$-matrices.
Writing
$$
B(\xi_0,\bxi):=\xi_0^2B^{00}+\xi_0C(\bxi)+B(\bxi),\quad \bxi=(\xi_1,...,\xi_d),
$$
with $C(\bxi)=C^j\xi_j, B(\bxi)=B^{jk}\xi_j\xi_k$,
we assume that \eqref{sos} is {\it semi-strictly definitely hyperbolic} in the sense that
\begin{itemize}
 \item $B^{00}<0$,
 \item for any $\bxi\in S^{d-1}$, the polynomial
    $$\xi_0\mapsto p^\bxi(\xi_0):=\det(B(\xi_0,\bxi))$$
    of degree $2n$
    has exclusively real non-vanishing zeroes $\xi_0$, 
    whose multiplicities do not depend on $\bxi$,
    and  
 \item for any such zero of multiplicity $\nu$\!, the {\it amplitude space}
    $$X(\xi_0,\bxi)=\ker B(\xi_0,\bxi)$$ has dimension~$\nu$. 
\end{itemize}
Note that the equation 
\be\label{drsos}
p^\bxi(\xi_0)=0
\ee
is the dispersion relation of system \eqref{sos}. 
\par\bigskip
With $U_t=:P, U_{x_j}=:Q_j$, system \eqref{sos} implies
\be\label{fos}
\begin{aligned}
B^{00}P_t+C^j P_{x^j}+B^{jk}(Q_j)_{x^k}&=0\\
(Q_j)_t-P_{x^j}&=0, 
%\\  U_t&=P
\end{aligned}
\ee
%with constraints 
%\be
%(Q_j)_{x^k}-(Q_k)_{x^j}=0.
%\ee
which we write as
\be\label{fosm}
\begin{pmatrix}
B^{00}&&&\\
&I&& \\
&&\ddots&\\
&&&I 
\end{pmatrix}
\begin{pmatrix}
P\\ Q_1\\ \vdots\\ Q_d 
\end{pmatrix}_t
+
\begin{pmatrix}
C^k&B^{1k}&\cdots&B^{dk}\\
-\delta_1^kI&0&\cdots&0\\
\vdots&\vdots&\ddots&\vdots\\
-\delta_d^kI&0&\cdots&0
\end{pmatrix}
\begin{pmatrix}
P\\ Q_1\\ \vdots\\ Q_d 
\end{pmatrix}_{x^k}
=0
\ee
or briefly
\be\label{fosb}
A^0V_t+A^kV_{x^k}=0.
\ee

\newpage

\begin{lemma}
(i) The dispersion relation of \eqref{fos} is
\be\label{qp}
0=q^\bxi(\xi_0):=\xi_0^{(d-1)n}p^\bxi(\xi_0)
%\xi_0^{(d-1)n}\det(\xi_0^2B^{00}+\xi_0C(\bxi)+B(\bxi)).
\ee
(ii) The non-zero modes $(\lambda(\bxi),R(\bxi)),\bxi\in S^{d-1},$ of the first-order formulation 
\eqref{fos}/\eqref{fosm}/\eqref{fosb}, i.e., the pairs $(\lambda(\bxi),R(\bxi))$  
satisfying $\lambda(\bxi)\neq 0$ and 
$$
\{0\}\neq R(\bxi)=\ker (-\lambda(\bxi) A^0+A(\bxi)),\quad A(\bxi)\equiv A^k\xi_k,
$$
are given by  
$$
(\lambda(\bxi),R(\bxi))=(-\xi_0,\bar X(\xi_0,\bxi)),\quad\bxi\in S^{d-1},
$$
with $\xi_0$ a zero of $p^\bxi$ and
$$
\bar X(\xi_0,\bxi)=\{\begin{pmatrix}
                     \xi_0x\\\xi_1x\\\vdots\\\xi_dx 
                     \end{pmatrix}:x\in X(\xi_0,\bxi)
                    \}.
$$
(iii) The zero mode is $(0,R_0(\bxi))$ with some $R(\bxi)$ of dimension $(d-1)n$.
\end{lemma}
\begin{proof}
(i) With $d>1$ one has
\be
\begin{aligned}
q^\bxi(\xi_0)=&\det(A^\alpha\xi_\alpha)=
\left|
\begin{matrix}
B^{00}\xi_0+C^k\xi_k&B^{1k}\xi_k&\cdots&B^{dk}\xi_k\\
-\xi_1 I&\xi_0I&\cdots&0\\
\vdots&\vdots&\ddots&\vdots\\
-\xi_d I&0&\cdots&\xi_0I\\
\end{matrix}
\right|\\
=&
\det(\xi_0I)
\left|
\begin{matrix}
B^{00}\xi_0+C^k\xi_k+\xi_0^{-1}\xi_d\xi_kB^{dk}&B^{1k}\xi_k&\cdots&B^{(d-1)k}\xi_k\\
-\xi_1 I&\xi_0I&\cdots&0\\
\vdots&\vdots&\ddots&\vdots\\
-\xi_{d-1} I&0&\cdots&\xi_0I
\end{matrix}
\right|\\
=&\hdots=
(\det(\xi_0I))^d
\left|
\begin{matrix}
B^{00}\xi_0+C^k\xi_k+\xi_0^{-1}\xi_j\xi_kB^{jk}
\end{matrix}
\right|\\
=&\xi_0^{dn}\det(\xi_0^{-1}B(\xi_0,\bxi))\\
=&\xi_0^{(d-1)n}p^\bxi(\xi_0)
\end{aligned}
\ee
thus \eqref{qp} for $\xi_0\neq0$, and by continuity also for $\xi_0=0$.
%The case 
$d=1$ %works 
analogously.\par\medskip\noindent
(ii) Obviously,
\be\label{eqker}
\begin{pmatrix}
B^{00}\xi_0+C^k\xi_k&B^{1k}\xi_k&\cdots&B^{dk}\xi_k\\
-\xi_1 I&\xi_0I&\cdots&0\\
\vdots&\vdots&\ddots&\vdots\\
-\xi_d I&0&\cdots&\xi_0I\\
\end{pmatrix}
\begin{pmatrix}
\bar x_0\\\bar x_1\\\vdots\\\bar x_d
\end{pmatrix}
=
\begin{pmatrix}
0\\0\\\vdots\\0 
\end{pmatrix}
\ee
is equivalent to $\bar x_\alpha=\xi_\alpha x$ with some $x\in X(\xi_0,\bxi)$.
\par
\goodbreak
(iii) The left kernel of $A(\bxi)$, i.e., $\{(0,l_1,...,l_d)\in\R^{(d+1)n}:\xi^kl_k=0\}$ has 
dimension $(d-1)n$, and thus also the right kernel $R(\bxi)$.
\end{proof}
%We occasionally refer to the roots $\xi_0$ of $p^\bxi$ as the physical speeds associated 
%with the spatial direction $\bxi\in S^{d-1}$.
\section{Quasisemilinear second-order systems}
From now on we consider nonlinear systems of the form 
\be\label{sosn}
B^{00}(U)U_{tt}+C^j(U)U_{tx^j}+B^{jk}(U)U_{x^jx^k}=H(U,U_t,U_x),
\ee
i.e., the differential operator is linear in the highest, second-order, derivatives, 
with coefficients depending only on the unknown functions itself, not its first-order derivatives.   
The unknown/solution $U$ ranges in some open subset $\mathcal U$ of $\R^n$.

We assume that the operator on the left-hand side of \eqref{sosn} is {\it pointwise  
semi-strictly definitely hyperbolic}, in the sense that every linear operator 
$$
U\mapsto 
B^{00}(U^*)U_{tt}+C^j(U^*)U_{tx^j}+B^{jk}(U^*)U_{x^jx^k},\quad U^*\in\mathcal U,
$$
is semi-strictly definitely hyperbolic
\par\bigskip
With again $Q=U_x$, system \eqref{sosn} has the first-order counterpart
\be\label{fosn}
\begin{aligned}
B^{00}(U)P_t+C^j(U)P_{x^j}+B^{jk}(U)(Q_j)_{x^k}&=H(U,P,Q)\\
(Q_j)_t-P_{x^j}&=0\\
U_t&=P,
\end{aligned}
\ee
%with constraints 
%\be
%(Q_j)_{x^k}-(Q_k)_{x^j}=0.
%\ee
which we write as
\be\label{fosnb}
\mA^0(\mV)\mV_t+\mA^k(\mV)\mV_{x^k}=\mH(\mV).
\ee
with
$$
\mV=(P,Q_1,...,Q_d,U),
$$
\small
$$ %\be\label{fosnm}
\mA^0(\mV)=
\begin{pmatrix}
B^{00}(U)&&&&\\
&I&&& \\
&&\ddots&&\\
&&&I&\\
&&&&I
\end{pmatrix},
\mA^k(\mV)=
\begin{pmatrix}
C^k(U)&B^{1k}(U)&\cdots&B^{dk}(U)&0\\
-\delta_1^kI&0&\cdots&0&0\\
\vdots&\vdots&\ddots&\vdots&\vdots\\
-\delta_d^kI&0&\cdots&0&0\\
0&0&\cdots&0&0
\end{pmatrix}
$$
\normalsize
and 
$$
 \mH(\mV)=(H(U,P,Q),0,...,0,P).
$$ 
The modes of \eqref{fosn}, i.e., pairs $(\lambda(\mV,\bxi),\mathcal R(\mV,\bxi))$ with
$$
\{0\}\neq \mR(\mV,\bxi)=\ker (-\lambda(\mV,\bxi) \mA^0(\mV)+\mA(\mV,\bxi)),
\quad 
\mA(\mV,\bxi)\equiv \mA^k(\mV)\xi_k,
$$
now depend on the state $\mV$, with $\mR(\mV,\bxi)$ consisting of all solutions 
$(\bar x_0,\bar x_1,...,\bar x_d,y)\in\R^{(d+1)n}$ of 
\be\label{eqkern} 
\begin{pmatrix}
B^{00}(U)\xi_0+C^k(U)\xi_k&B^{1k}(U)\xi_k&\cdots&B^{dk}(U)\xi_k&0\\
-\xi_1 I&\xi_0I&\cdots&0&0\\
\vdots&\vdots&\ddots&\vdots&\vdots\\
-\xi_d I&0&\cdots&\xi_0I&0\\
0&0&\cdots&0&\xi_0I
\end{pmatrix}
\begin{pmatrix}
\bar x_0\\\bar x_1\\\vdots\\\bar x_d\\y
\end{pmatrix}
=
\begin{pmatrix}
0\\0\\\vdots\\0\\0 
\end{pmatrix}
\ee
with $-\lambda=\xi_0$ equal to $0$ or a zero of $p^\bxi_U=\ker B(U,.,\bxi)$.  

In view of the close proximity of \eqref{eqkern} and \eqref{eqker},
the non-zero modes are of the form 
\be\label{nonzeromode}
\mR(\mV,\bxi)=R(U,\bxi)\times\{0\}, 
\ee
while 
$$
\mR_0(\mV,\bxi)=(R_0(U,\bxi)\times\{0\})+(\{0\}\times\R^n).
$$
Following Lax \cite{L57}, a smooth mode $(\lambda(\mV,\bxi),\mR(\mV,\bxi))$ of a hyperbolic system 
\eqref{fosnb} is called {\it linearly degenerate} if the directional derivative, in state space,
of the characteristic speed $\lambda$ with respect to the state $\mV$ vanishes along the associated 
amplitude space $\mR(\mV,\bxi)$.
%its multiplicity is 1, i.e. $\mR(\mV,\bxi)$ is given by a vector field, and 
$$
\mR(\mV,\bxi)\cdot\frac\partial{\partial \mV}(\lambda(\mV,\bxi))=0.
$$
The point of this note is the following.
\begin{prop}
All modes of the first order formulation \eqref{fosn} of the quasi\-semilinear system 
\eqref{sosn} are linearly degenerate.
\end{prop}
\begin{proof}
For the zero mode, the linear degeneracy trivial. For the non-zero modes it follows from 
\eqref{nonzeromode} and the fact that for every $\bxi\in S^{d-1}$, 
$\lambda(.,\bxi)$ is a function of $U$ alone. 
\end{proof}
\begin{coro}
The first-order versions of all models of dissipative relativistic fluid dynamics considered in 
\cite{FT14,FT17,FT18,BDN18,BDN19,FS24} have exclusively linearly degenerate modes.     
\end{coro}
Could it be that solutions to any of these models for smooth data generally avoid singularity 
formation?

\end{document}